\documentclass[12pt, reqno]{amsart}
\usepackage{amsmath, amsthm, amscd, amsfonts, amssymb, graphicx, color}
\usepackage[bookmarksnumbered, colorlinks, plainpages]{hyperref}
\hypersetup{colorlinks=true,linkcolor=red, anchorcolor=green, citecolor=cyan, urlcolor=red, filecolor=magenta, pdftoolbar=true}

\newtheorem{theorem}{Theorem}[section]
\newtheorem{lemma}[theorem]{Lemma}
\newtheorem{proposition}[theorem]{Proposition}
\newtheorem{corollary}[theorem]{Corollary}

\theoremstyle{definition}
\newtheorem{definition}[theorem]{Definition}
\newtheorem{example}[theorem]{Example}

\theoremstyle{remark}

\numberwithin{equation}{section}

\usepackage[all]{xy}
\usepackage{amscd}

\newcommand{\pr}{\mathrm{pr}}

\newcommand{\eps}{\varepsilon }
\newcommand{\e}{\varepsilon}

\newcommand{\R}{\mathbb R}
\newcommand{\IR}{\mathbb R}
\newcommand{\N}{\mathbb N}

\newcommand{\U}{\mathbb U}
\newcommand{\K}{\mathfrak K}
\newcommand{\catB}{\mathfrak{B}}
\newcommand{\BI}{\mathfrak{BI}}
\newcommand{\DL}{\mathfrak L}

\newcommand{\FI}{\mathfrak{FI}}
\newcommand{\RFI}{\mathfrak{RI}}
\newcommand{\RI}{\mathfrak{RI}}
\newcommand{\w}{\omega}

\newcommand{\bas}{\mathrm{\ss}}
\newcommand{\conv}{\mathrm{conv}}

\begin{document}

\setcounter{page}{1}

\title[A universal Banach space with unconditional basis]{A universal Banach space\\ with a $K$-unconditional basis}

\author[T. Banakh \MakeLowercase{and} J. Garbuli\'nska-Wegrzyn]
{Taras Banakh$^{1,2}$ \MakeLowercase{and} Joanna Garbuli\'nska-W\c egrzyn$^{*1}$}

\address{$^{1}$ Jan Kochanowski University, Kielce, Poland}
\address{$^{2}$Ivan Franko National University of Lviv, Ukraine}
\email{\textcolor[rgb]{0.00,0.00,0.84}{t.o.banakh@gmail.com, jgarbulinska@ujk.edu.pl}}


\let\thefootnote\relax\footnote{Copyright 2016 by the Tusi Mathematical Research Group.}



\date{Received: xxxxxx; Revised: yyyyyy; Accepted: zzzzzz.
\newline \indent $^{*}$Corresponding author}

\begin{abstract}
For a constant $K\geq 1$ let $\catB_K$ be the class of pairs $(X,(\mathbf e_n)_{n\in\w})$ consisting of a Banach space $X$ and an unconditional Schauder basis $(\mathbf e_n)_{n\in\w}$ for $X$, having the unconditional basic constant $K_u\le K$. Such pairs are called $K$-based Banach spaces. A based Banach space $X$ is rational if the unit ball of any finite-dimensional subspace spanned by finitely many basic vectors is a polyhedron whose vertices have rational coordinates in the Schauder basis of $X$.

Using the technique of Fra\"iss\'e theory, we construct a rational $K$-based Banach space $\big(\U_K,(\mathbf e_n)_{n\in\w}\big)$ which is $\RI_K$-universal  in the sense that each basis preserving isometry $f:\Lambda\to\U_K$ defined on a based subspace $\Lambda$ of a finite-dimensional rational $K$-based Banach space $A$ extends to a basis preserving isometry $\bar f:A\to\U_K$ of the based Banach space $A$. 
We also prove that the $K$-based Banach space $\U_K$ is almost $\FI_1$-universal in the sense that any  base preserving
$\e$-isometry $f:\Lambda\to\U_K$ defined on a based subspace $\Lambda$ of a finite-dimensional $1$-based Banach space $A$ extends to a base preserving $\e$-isometry $\bar f:A\to\U_K$ of the based Banach space $A$. On the other hand, we show that no almost $\FI_K$-universal based Banach space exists for $K>1$.

The Banach space $\U_K$ is isomorphic to the complementably universal Banach space for the class of Banach spaces with an unconditional Schauder basis, constructed by Pe\l czy\'nski in 1969. \end{abstract} 
\maketitle

\parskip1pt

\section{Introduction}
A Banach space $X$ is ({\em complementably}\/) {\em universal} for a given class of Banach spaces if $X$ belongs to this class and every space from the class is isomorphic to a (complemented) subspace of $X$.

In 1969 Pe\l czy\'nski \cite{pelbases} constructed a complementably universal Banach space for the class of Banach spaces with a Schauder basis.
In 1971 Kadec \cite{kadec} constructed a complementably universal Banach space for the class of spaces with the \emph{bounded approximation property} (BAP).
In the same year Pe\l czy\'nski \cite{pel_any} showed that every Banach space with BAP is isomorphic to a complemented subspace of a Banach space with a basis.
Pe\l czy\'nski and Wojtaszczyk  \cite{wojtaszczyk} constructed in 1971 a (complementably) universal Banach space for the class of spaces with a finite-dimensional decomposition.
Applying Pe\l czy\'nski's decomposition argument \cite{pelczynski}, one immediately concludes that all three universal spaces are isomorphic.
It is worth mentioning a negative result of Johnson and Szankowski \cite{JohnSzan} saying that no separable Banach space can be complementably universal for the class of all separable Banach spaces.

In \cite{asia} the second author constructed an isometric version of the Kadec-Pe\l czy\'nski-Wojtaszczyk space.  The universal Banach space from \cite{asia} was constructed using the general categorical technique of Fra\"iss\'e limits \cite{kubis}. 
In \cite{taras}, for every $K\ge 1$ the authors constructed a universal space for the class of rational Banach spaces endowed with a normalized unconditional Schauder basis with a suppression constant $K_s\le K$. The constructed space is isomorphic to the complementably universal space for Banach spaces with unconditional basis, which was constructed by Pe\l czy\'nski in \cite{pelbases}.
By definition, the \emph{suppression constant} of an unconditional Schauder basis $(\mathbf{e}_n)_{n\in \omega}$ of a Banach space $X$ is the smallest real number $K_s$ such that $$\Big\| \sum_{n=0}^{\infty} \eps_n \cdot x_n\cdot \mathbf{e}_n\Big\| \leq K_s\cdot \| x\|$$ for any $x=\sum_{n=0}^{\infty}x_n\cdot \mathbf{e}_n\in X$ and any sequence $(\eps_n)_{n\in \omega} \in \{0,1\}^\omega$. 
A Banach space $X$ with a Schauder basis $(\mathbf e_n)_{n\in\w}$ is {\em rational} if for any finite set $F\subset \w$ the unit ball of the finite-dimensional subspace of $X$ spanned by the set $\{\mathbf e_n\}_{n\in F}$ coincides with a polyhedron whose vertices have rational coordinates in the basis $(\mathbf e_n)_{n\in F}$.

In this paper we apply the categorical method of Fra\"iss\'e limits for constructing a universal space $\U_K$ in the class of rational Banach spaces endowed with a normalized unconditional Schauder basis with unconditional constant $K_u\le K$, where $K$ is any real number $\ge 1$. By definition, the \emph{unconditional constant} of an unconditional Schauder basis $(\mathbf{e}_n)_{n\in \omega}$ of a Banach space $X$ is the smallest real number $K_u$ such that $$\Big\| \sum_{n=0}^{\infty} \eps_n \cdot x_n\cdot \mathbf{e}_n\Big\| \leq K_u\cdot \| x\|$$ for any $x=\sum_{n=0}^{\infty}x_n\cdot \mathbf{e}_n\in X$ and any sequence $(\eps_n)_{n\in \omega} \in \{-1,1\}^\omega$. It is known \cite[3.1.5]{AK} that $K_s\leq K_u \leq 2 K_s$. We shall also prove that the space $\U_K$ for every $\e>0$ contains an $\e$-isometric copy of any Banach space with 1-unconditional Schauder basis.

The universal space $\U_K$ (for rational Banach spaces with unconditional Schauder basis with unconditional constant $K_u\leq K$) constructed in this paper is isomorphic (but not isometric) to the universal space (for rational Banach spaces with unconditional Schauder basis with suppression constant $K_s\leq K$) built in \cite{taras}. In fact, the construction of the universal space $\U_K$ exploits the standard technique of Fra\"iss\'e limits \cite{kubis} and is similar to the constructions from \cite{asia} and especially from \cite{taras}. So, in the proofs of some statements we refer the reader to \cite{taras} and provide full proofs of lemmas that differ from the corresponding lemmas in \cite{taras}.

\section{Preliminaries}
All Banach spaces considered in this paper are separable and are over the field $\R$ of real numbers.
\subsection{Definitions}

Let $X$ be a Banach space with a Schauder basis $(\mathbf e_n)_{n=1}^\infty$ and let $(\mathbf e_n^*)_{n=1}^\infty$ be the corresponding sequence of coordinate functionals.
The basis $(\mathbf e_n)_{n=1}^\infty$ is called {\em $K$-unconditional} for a real constant $K\ge 1$ if $K_u \leq K$.
Since each $K$-unconditional Schauder basis $(\mathbf e_n)_{n=1}^\infty$ is unconditional, for any $x\in X$ and any permutation $\pi$ of $\N$ the series $\sum_{n=1}^\infty \mathbf e_{\pi(n)}^*(x)\cdot \mathbf e_{\pi(n)}$ converges to $x$. This means that we can forget about the ordering and think of a $K$-unconditional basis of a Banach space as a subset $\bas\subset X$ such that for some bijection $\mathbf e:\mathbb N\to \bas$ the sequence $(\mathbf e(n))_{n=1}^\infty$ is a $K$-unconditional Schauder basis for $X$.

More precisely, by a {\em normalized $K$-unconditional basis} for a Banach space $X$ we shall understand a subset $\bas\subset X$ for which there exists a family $\{\mathbf e^*_b\}_{b\in \bas}\subset X^*$ of continuous functionals such that
\begin{itemize}
\item $\|b\|=1=\mathbf e^*_b(b)$ for any $b\in \bas$;
\item $\mathbf e_b^*(b')=0$ for every $b\in \bas$ and $b'\in \bas\setminus \{b\}$;
\item $x=\sum_{b\in \bas}\mathbf e^*_b(x)\cdot b$ for every $x\in X$;
\item for any function $\pm:\bas\to\{-1,1\}$ the operator $T_\pm:X\to X$, $T_\pm:x\mapsto \sum_{b\in \bas}\pm(b)\cdot\mathbf e^*_b(x)\cdot b$, is well-defined and has norm $\|T_\pm\|\le K$.
\end{itemize}
The equality $x=\sum_{b\in \bas}\mathbf e^*_b(x)\cdot b$ in the third item means that for every $\eps>0$ there exists a finite subset $F\subset \bas$ such that $\|x-\sum_{b\in E}\mathbf e^*_b(x)\cdot b\|<\eps$ for every finite subset $E\subset \bas$ containing $F$.

By a \emph{$K$-based Banach space} we shall understand a pair $(X,\bas_X)$ consisting of a Banach space $X$ and a normalized $K$-unconditional basis $\bas_X$ for $X$. A {\em based Banach space} is a pair $(X,\bas_X)$, which is a $K$-based  Banach space for some $K\ge 1$. A based Banach space $(X,\bas_X)$ is a \emph{subspace} of a based Banach space $(Y,\bas_Y)$ if $X\subseteq Y$ and $\bas_X=X\cap \bas_Y$.

For a Banach space $X$ by $\|\cdot\|_X$ we denote the norm of $X$ and by $$B_X:=\{x\in X:\|x\|_X\le 1\}$$ the closed unit ball of $X$.

A finite-dimensional based Banach space $(X,\bas_X)$ is \emph{rational} if its unit ball is a convex polyhedron spanned by finitely many vectors whose coordinates in the basis $\bas_X$ are rational. 
A based Banach space $(X,\bas_X)$ is \emph{rational} if each finite dimensional $K$-based subspace of $(X,\bas_X)$ is rational.

\subsection{Categories}
Let $\K$ be a category. For two objects $A,B$ of the category $\K$, by $\K(A,B)$ we denote the set of all $\K$-morphisms from $A$ to $B$. A \emph{subcategory} of $\K$ is a category $\DL$ such that each object of $\DL$ is an object of $\K$ and each morphism of $\DL$ is a morphism of $\K$. Morphisms and isomorphisms of a category $\K$ will be called {\em $\K$-morphisms} and {\em $\K$-isomorphisms}, respectively.

A subcategory $\mathfrak L$ of a category $\K$ is {\em full} if each $\K$-morphism between objects of the category $\mathfrak L$ is an $\mathfrak L$-morphism.

A category $\DL$ is \emph{cofinal} in $\K$ if for every object $A$ of $\K$ there exists an object $B$ of $\DL$ such that the set $\K(A,B)$ is nonempty. A category $\K$ has the {\it amalgamation property} if for every objects $A, B, C$ of $\K$ and for every morphisms $f\in \K(A,B)$, $g\in \K(A,C)$ we can find an object $D$ of $\K$ and morphisms $f'\in \K(B,D)$, $g'\in \K(C,D)$ such that $f'\circ f=g'\circ g$.

In this paper we shall work in the category $\catB$, whose objects are  based Banach spaces. For two based Banach spaces $(X,\bas_X)$, $(Y,\bas_Y)$, a morphism of category $\catB$ is a linear continuous operator $T: X\to Y$ such that $T(\bas_X)\subseteq \bas_Y$. 

A morphism $T:X\to Y$ of the category $\catB$ is called an \emph{isometry} (or else an {\em isometry morphism}) if $\|T(x)\|_Y=\|x\|_X$ for any $x\in X$.  By $\BI$ we denote the category whose objects are based Banach spaces and morphisms are isometry morphisms of based Banach spaces. The category $\BI$ is a subcategory of the category $\catB$. 

For any real number $K\ge 1$ let $\catB_K$ (resp. $\BI_K$) be the category whose objects are $K$-based Banach spaces and morphisms are (isometry) $\catB$-morphisms between $K$-based Banach spaces. So, $\catB_K$ and $\BI_K$ are full subcategories of the categories $\catB$ and $\BI$, respectively.  

By $\FI_K$ we denote the full subcategory of $\BI_K$, whose objects are finite-dimensional  $K$-based Banach spaces, and by $\RFI_K$ the full subcategory of $\FI_K$  whose objects are rational finite-dimensional  $K$-based Banach spaces.
So, we have the inclusions $\RFI_K\subset\FI_K\subset\BI_K$ of categories.

From now on we assume that $K\ge 1$ is some fixed real number.

\subsection{Amalgamation}
In this section we prove that the categories $\FI_K$ and $\RI_K$ have the amalgamation property.

\begin{lemma}[\bf Amalgamation Lemma]\label{lemat} Let $X$, $Y$, $Z$ be finite-dimensional $K$-based Banach spaces and $i:Z\to X$, $j:Z \to Y$ be $\BI_K$-morphisms. Then there exist a finite-dimensional $K$-based Banach space $W$ and $\BI_K$-morphisms $i':X\to W$ and $j':Y\to W$ such that  the diagram
$$\xymatrix{
Y \ar[r]^{j'} &W \\
Z \ar[r]^i \ar[u]^j & X\ar[u]^{i'}}$$
is commutative.\\
Moreover, if the $K$-based Banach spaces $X$, $Y$, $Z$ are rational, then so is the $K$-based Banach space $W$.
\end{lemma}

\begin{proof}
Consider the direct sum $X\oplus Y$ of the Banach spaces $X$, $Y$ endowed with the norm $\| (x,y)\| =\| x\| _X+\| y\| _Y$.
Let $W=(X\oplus Y)/\Delta $ be the quotient by the subspace $\Delta =\{(i(z),-j(z)):z\in Z\}$.

We define linear operators $i':X\to W$ and $j':Y\to W$ by $i'(x)=(x,0)+\Delta $ and $j'(y)=(0,y)+\Delta .$\\
Let us show $i'$ and $j'$ are isometries. Indeed, for every $x\in X$
\begin{align*}
\| i'(x)\|_W &=\operatorname{dist}((x,0),\Delta )\leq \| (x,0)\| =\| x\| _X+\| 0\| _Y=\| x\| _X.
\end{align*}
On the other hand, for every $z\in Z$
\begin{align*} 
&\| (x,0)-(i(z),-j(z))\| =\| (x-i(z),j(z))\| =\| x-i(z)\| _X+\| j(z)\| _Y=\\
&=\|  x-i(z)\| _X+\|z\|_Z=\|  x-i(z)\| _X+\| i(z)\| _X\geq \| x-i(z)+i(z)\| _X=\| x\| _X
\end{align*}
and hence $\|x\|_X \leq \inf_{z\in Z} \| (x,0)-(i(z),-j(z))\|=\|i'(x)\|_W$. Therefore $\|i'(x)\|_W=\|x\|_X$.
Similarly, we can show that $j'$ is an isometry.

Since $i:Z\to X$ and $j:Z\to Y$ are $\BI_K$-morphisms, for every $b\in \bas_Z$ we have $i(b)\in\bas_X$ and $j(b)\in\bas_Y$ and $$i'\circ i(b)=(i(b),0)+\Delta\ni (i(b),0)-(i(b),-j(b))=(0,j(b))\in (0,j(b))+\Delta=j'\circ j(b),$$
which implies $i'\circ i(b)=j'\circ j(b)$ and hence $i'\circ i=j'\circ j$.  

We shall identify $X$ and $Y$ with their images $i'(X)$ and $j'(Y)$ in $W$ and $Z$ with its image $i'\circ i(Z)=j'\circ j(Z)$ in $W$. In this case the union $\bas_W:=\bas_X\cup \bas_Y$ is a Schauder basis for the (finite-dimensional)  Banach space $W$ and $\bas_X\cap\bas_Y=\bas_Z$. Since the embeddings $i,j,i',j'$ are isometries, each vector $b\in\bas_W$ has norm one. 

Next, we show that the basis $\bas_W$ has unconditional constant $\le K$. Given any $w\in W$ we should prove the upper bound 
$$\big\| \sum _{b\in \bas_W} \pm_b\cdot w_b\cdot b \big\| \leq K\cdot\big\| \sum _{b\in \bas_W} w_b\cdot b\big\|,$$ 
where $\pm_b\in \{-1,1\}$ and $(w_b)_{b\in \bas_W}$ are the coordinates of $w=\sum_{b\in \bas_W}w_b \cdot b$ in the basis $\bas_W$.

Taking into account that the bases $\bas_X$ and $\bas_Y$ of the $K$-based Banach spaces $X$, $Y$ have unconditional constant $\le K$, we obtain: 
{\small
\begin{align*}
\big\| \sum _{b\in \bas_W}& \pm_b w_b b\big\|_W =\inf \Big\{ \big\| \sum _{b\in \bas_X\backslash  \bas_Z} \pm_b w_b b+\sum _{b \in \bas_Z}w'_b b\big\|_X+\big\|\sum _{b \in \bas_Z}w''_b b +\sum _{b\in \bas_Y\backslash  \bas_Z} \pm_b w_b b \big\|_Y: \\
&(w'_b)_{b\in \bas_Z},(w''_b)_{b\in \bas_Z}\in\mathbb R^{\bas_Z}, \quad \forall  b\in \bas_Z \;\;w'_b+w''_b=\pm_b w_b\Big\}= \\
=&\inf \Big\{ \big\| \sum _{b\in \bas_X\backslash  \bas_Z} \pm_b w_b b+\sum _{b \in \bas_Z}\pm_bw'_b b\big\|_X+\big\|\sum _{b \in \bas_Z}\pm_bw''_b b +\sum _{b\in \bas_Y\backslash  \bas_Z} \pm_b w_b b \big\|_Y: \\&
(w'_b)_{b\in \bas_Z},(w''_b)_{b\in \bas_Z}\in\mathbb R^{\bas_Z}, \quad \forall  b\in \bas_Z \;\; w'_b+w''_b=w_b\Big\} \leq \\&
\leq K\cdot\inf \Big\{ \big\| \sum _{b\in \bas_X\backslash  \bas_Z} w_b b+\sum _{b \in \bas_Z}w'_b b\big\|_X+\big\|\sum _{b \in \bas_Z}w''_b b +\sum _{b\in \bas_Y\backslash  \bas_Z}  w_b b \big\|_Y: \\&
(w'_b)_{b\in \bas_Z},(w''_b)_{b\in \bas_Z}\in\mathbb R^{\bas_Z}, \quad \forall  b\in \bas_Z \;\; w'_b+w''_b= w_b\Big\} =K\cdot\big\| \sum _{b\in \bas_W} w_b b\big\|_W.
\end{align*}
}

If the finite-dimensional based Banach spaces $X$ and $Y$ are rational, then so is their sum $X\oplus Y$ and so is the quotient space $W$ of $X\oplus Y$.
\end{proof}

\section{$\catB$-universal based Banach spaces}

\begin{definition}\label{d:uB} A based Banach space $U$ is defined to be {\em $\catB$-universal} if each based Banach space $X$ is $\catB$-isomorphic to a based subspace of $U$.
\end{definition}

Definition~\ref{d:uB} implies that each $\catB$-universal based Banach space is complementably universal for the class of Banach spaces with unconditional basis.
 Reformulating Pe\l czy\'nski's Uniqueness Theorem 3 \cite{pelbases}, we obtain the following uniqueness result.

\begin{theorem}[Pe\l czy\'nski]\label{t:up} Any two $\catB$-universal based Banach spaces are $\catB$-isomorphic.
\end{theorem}

A $\catB$-universal based Banach space $\U$ was constructed by Pe\l czy\'nski in \cite{pelbases}. In the following sections we shall apply the technique of Fra\"iss\'e limits to construct many $\catB$-isomorphic copies of the Pe\l czy\'nski's $\catB$-universal space $\U$.

\section{$\RI_K$-universal based Banach spaces}


\begin{definition}\label{def1}
A based Banach space $X$ is called \emph{$\RI_K$-universal} if for any rational finite-dimensional $K$-based Banach space $A$, any isometry morphism $f:\Lambda\to X$ defined on a based subspace $\Lambda$ of $A$ can be extended to an isometry morphism $\bar f:A\to X $.
\end{definition}

We recall that $\RI_K$ denotes the full subcategory of $\BI$ whose objects are rational finite-dimensional $K$-based Banach spaces.
Obviously, up to $\RI_K$-isomorphism the category $\RI_K$ contains countably many objects.
By Lemma~\ref{lemat}, the category $\RI_K$ has the amalgamation property.
We now use the concepts from \cite{kubis} for constructing a ``generic" sequence in $\RI_K$.
A sequence $(X_n)_{n\in \omega}$ of objects of the category $\BI_K$ is called \emph{a chain} if each $K$-based Banach space $X_n$ is a subspace of the $K$-based Banach space $X_{n+1}$.  

\begin{definition}\label{Fresse}
A chain $(U_n)_{n\in \omega}$ of objects of the category $\RI_K$ is \emph{Fra\"iss\'e} if for any $n\in \omega$ and $\RI_K$-morphism $ f: {U_n} \to Y$ there exist $m > n$ and an $\RI_K$-morphism  $g: Y\to {U_m}$ such that $g \circ f:U_n\to U_m$ is the identity inclusion of $U_n$ to $U_m$.
\end{definition}

Definition \ref{Fresse} implies that the Fra\"iss\'e sequence $\{U_n\}_{n\in \omega}$ is cofinal in the category $\RI_K$ in the sense that each object $A$ of the category $\RI_K$ admits an $\RI_K$-morphism $A\to U_n$ for some $n\in \omega$. In this case the category $\RI_K$ is countably cofinal.

The name ``Fra\"iss\'e sequence", as in \cite{kubis}, is motivated by the model-theoretic theory of Fra\"iss\'e limits developed by Roland Fra\"iss\'e \cite{fra}.
One of the results in \cite{kubis} is that every countably cofinal category with amalgamation has a Fra\"iss\'e sequence.
Applying this general result to our category $\RI_K$ we get:

\begin{theorem}[\cite{kubis}]\label{Kubisia}
The category $\RI_K$ has a Fra\"iss\'e sequence.
\end{theorem}

From now on, we fix a Fra\"iss\'e\ sequence $(U_n)_{n\in \omega}$ in $\RI_K$.
Let $\U_K$ be the completion of the union $\bigcup_{n\in \omega}U_n$ and $\bas_{\U_K}=\bigcup_{n\in \omega}\bas_{U_n}\subset\U_K$. The proof of the following lemma literally repeats the proof of Lemma~4.4 in \cite{taras}.

\begin{lemma}
$(\U_K,\bas_{\U_K})$ is an $\RI_K$-universal rational $K$-based Banach space.
\end{lemma}

To shorten notation, the $\RI_K$-universal rational $K$-based Banach space $(\U_K,\bas_{\U_K})$ will be denoted by $\U_K$. The following theorem  shows that such space is unique up to $\BI$-isomorphism.

\begin{theorem}
Any $\RI_K$-universal rational $K$-based Banach spaces $X$, $Y$ are $\BI$-isomorphic, which means that there exists a linear bijective isometry $X\to Y$ preserving the bases of $X$ and $Y$.
\end{theorem}

This theorem can be proved by analogy with Theorem 4.5 of \cite{taras}.

\section{Almost $\FI_K$-universality}

By analogy with the $\RI_K$-universal based Banach space, one can try to introduce an $\FI_K$-universal based Banach space. However such notion is vacuous as each based Banach space has only countably many finite-dimensional based subspaces whereas the category $\FI_K$ contains continuum many pairwise non-$\BI$-isomorphic  2-dimensional based Banach spaces. One of possibilities to overcome this problem is to introduce almost $\FI_K$-universal based Banach spaces with the help of $\eps$-isometries.

A linear operator $f$ between Banach spaces $X$ and $Y$ is called an \emph{$\eps$-isometry} for a positive real number $\eps$, if
$$ (1+\eps )^{-1} \cdot \|x\|_X <  \|f(x)\|_Y < (1+\eps ) \cdot \|x\|_X$$
for every $x\in X\backslash \{0\}$. This definition implies that each $\eps$-isometry is an injective linear operator.

A morphism  of the category $\catB$ of based Banach spaces is called an {\em $\eps$-isometry $\catB$-morphism} if it is an $\eps$-isometry of the underlying Banach spaces.

\begin{definition}\label{def2}
A based Banach space $X$ is called \emph{almost $\FI_K$-universal} if for any $\eps>0$ and finite dimensional $K$-based Banach space $A$, any $\eps$-isometry $\catB$-morphism $f:\Lambda\to X$ defined on a based subspace $\Lambda$ of $A$ can be extended to an $\eps$-isometry $\catB$-morphism $\bar f:A\to X $.
\end{definition}

It is clear that each almost $\FI_K$-universal based Banach space is almost $\FI_1$-universal.

The following uniqueness theorem can be proved by analogy with Theorem 5.3 of \cite{taras}.

\begin{theorem}\label{glowne}
Let $X$ and $Y$ be almost $\FI_K$-universal $K$-based Banach spaces and $ \eps >0$. Each $\eps$-isometry $\catB$-morphism $f:X_0\rightarrow Y$ defined on a finite-dimensional based subspace $X_0$ of the $K$-based Banach space $X$ can be extended to an $\eps$-isometry $\catB$-isomorphism $\bar{f}:X \rightarrow Y$.
\end{theorem}

\begin{corollary} For any almost $\FI_K$-universal $K$-based Banach spaces
 $X$ and $Y$ and any $ \eps >0$ there exists an $\eps$-isometry $\catB$-isomorphism $f:X \rightarrow Y$.
\end{corollary}

The following universality theorem can be proved by analogy with Theorem 5.5 of \cite{taras}.

\begin{theorem}\label{wazne}  Let $U$ be an almost $\FI_K$-universal based Banach space. For any $\eps>0$ and any $K$-based Banach space $X$ there exists  an $\eps$-isometry $\catB$-morphism $f:X\to U$.
\end{theorem}

\begin{corollary}\label{przenor} Each almost $\FI_1$-universal based Banach space $U$ is $\catB$-universal.
\end{corollary}

\begin{proof} Given a based Banach space $X$, we need to prove that $X$ is $\catB$-isomorphic to a based subspace of $U$. For every $F\subset \bas_{X}$ consider the linear operator $T_F:X\to X$ defined by the formula $T_F(\sum _{b\in \bas_{X}} x_b \cdot b)=\sum _{b\in F} x_b \cdot b - \sum _{b\in \bas_X\setminus F} x_b \cdot b$ for $(x_b)_{b\in\bas_X}\in\mathbb R^{\bas_X}$. Denote by $X_1$ the based Banach space $X$ endowed with the equivalent norm
$$\|x\|_1:=\sup_{F\subset \bas_X}\|T_F(x)\|.$$

It is easy to check that $X_1$ is a 1-based Banach space. By Theorem~\ref{wazne}, for $\eps=\frac12$ there exists an $\eps$-isometry $\catB$-morphism $f:X_1\to U$. Then $f$ is a $\catB$-isomorphism between $X$ and the based subspace $f(X)=f(X_1)$ of the based Banach space $U$.
\end{proof}

Corollary~\ref{przenor} combined with the Uniqueness Theorem \ref{t:up} of Pe\l czy\'nski implies

\begin{corollary}\label{kun} Each almost $\FI_1$-universal based Banach space  is $\catB$-isomorphic to the $\catB$-universal space $\U$ of Pe\l czy\'nski.
\end{corollary}

Now we discuss the problem of existence of almost $\FI_K$-universal based Banach spaces. To our surprise we discovered that such spaces exist only for $K=1$. In the proof of this non-existence result we shall use the following example.

\begin{example}\label{ex} For any positive rational numbers $\eta<\e<\delta$ with $\e\le 1$ there exists a rational 3-dimensional $(1+2\delta)$-based Banach space $(A,\|\cdot\|_A)$ with basis $\{e_1,e_2,e_3\}$ and a $1$-unconditional norm $\|\cdot\|'_\Lambda$ on the based subspace $\Lambda\subset A$ spanned by the vectors $e_1,e_2$ such that 
\begin{enumerate}
\item $\|x_1e_1+x_2e_2\|_A=\max\{|x_1|,|x_2|\}$ for any $x_1,x_2\in\IR$;
\item $\|x_1e_1+x_2e_2\|'_\Lambda=\max\{|x_1|,|x_2|,\frac{1+\e}2(|x_1|+|x_2|)\}$ for any $x_1,x_2\in\IR$;
\item the point $a:=(1+\delta)(e_1+e_2)-\delta e_3$ has norm $\|a\|_A=\|e_3\|_A=1$;
\item for any norm $\|\cdot\|'_A$ on $A$ with $\|e_3\|'_A=1$ and  $\|e_1+e_2\|'_A>\frac1{1+\eta}\|e_1+e_2\|'_\Lambda$, the point $a=(1+\delta)(e_1+e_2)-\delta e_3$ has norm $\|a\|'_A\ge (1+\delta)(\tfrac{1+\e}{1+\eta}-\tfrac{\delta}{1+\delta}\big)$.
\end{enumerate}
\end{example}

\begin{proof} Let $A$ be a 3-dimensional linear space with basis $\{e_1,e_2,e_3\}$. Given any positive rational numbers $\eta<\e<\delta$ with $\e\le 1$, consider the point $$a:=(1+\delta)\cdot (e_1+e_2)-\delta e_3\in A$$ and observe that $\frac1{1+\delta}a+\frac{\delta}{1+\delta}e_3=e_1+e_2$. Let $P=\{e_1,e_2,e_3,e_1-e_2,a\}$ and $B$ be the convex hull of the set $P\cup(-P)$. The convex symmetric set $B$ coincides with the closed unit ball of some norm $\|\cdot\|_A$ on $A$. Let us show that $(A,\|\cdot\|_A)$ is a $(1+2\delta)$-based Banach space.

First we check that $\|e_i\|_A=1$ for $i\in\{1,2,3\}$. The inequality $\|e_i\|_A\le 1$ follows from the inclusion $e_i\in P\subset B$. To show that $\|e_i\|_A=1$, it suffices to find a linear functional $f_i:A\to\IR$ such that $f_i(e_i)=1=\max f_i(B)$. Let $e_1^*,e_2^*,e^*_3$ be the coordinate functionals of the basis $e_1,e_2,e_3$.
If $i\in\{1,2\}$, then consider the functional $f_i=e_i^*+e^*_3$. For this functional, we get $f_i(e_i)=f_i(e_3)=1$ and $f_i(P\cup(-P))\subset \{-1,0,1\}$ and hence $f_i(B)\subset[-1,1]$, which implies $\|e_i\|_A=\|e_3\|_A=1$. 

Since $a\in P\subset B$, the element $a$ has norm $\|a\|_A\le 1$. To see that $\|a\|_A=1$, consider the linear functional $f=\frac12(e_1^*+e_2^*)+e_3^*$ and observe that $f(a)=1=\max f(B)$.

It is easy to see that $$B\subset (1+2\delta){\cdot}\conv(\{e_1+e_2,-e_1-e_2,e_1-e_2,e_2-e_1,e_3,-e_3\})\subset (1+2\delta)B,$$ which implies that basis $\{e_1,e_2,e_3\}$ has unconditional constant $\le(1+2\delta)$ in the norm $\|\cdot\|_A$ and means that $(A,\|\cdot\|_A)$ is a $(1+2\delta)$-based Banach space.

Let $\Lambda$ be the based subspace of $A$, spanned by the basic vectors $e_1$ and $e_2$. It is easy to see that $B\cap \Lambda$ coincides with the convex hull of the set $\{e_1+e_2,e_1-e_2,-e_1+e_2,-e_1-e_2\}$, which implies that $\|x_1e_1+x_2e_2\|_A=\max\{|x_1|,|x_2|\}$ for any $x_1e_1+x_2e_2\in\Lambda$.

Now consider the norm $\|\cdot\|'_\Lambda$ on $\Lambda$ defined by $$\|x_1e_1+x_2e_2\|'_\Lambda=\max\{|x_1|,|x_2|,\tfrac{1+\e}2(|x_1|+|x_2|)\}.$$ Observe that the Banach space $(\Lambda,\|\cdot\|'_\Lambda)$ endowed with the base $\{e_1,e_2\}$ is a $1$-based Banach space and $\|x\|_A\le\|x\|_\Lambda'\le (1+\e)\|x\|_A$ for any $x\in\Lambda$.

Now take any norm $\|\cdot\|'_A$ on $A$ such that $\|e_3\|'_A=1$ and $$\|e_1+e_2\|'_A>\tfrac1{1+\eta}\|e_1+e_2\|'_\Lambda=\tfrac{1+\e}{1+\eta}.$$ Then
$$\tfrac{1+\e}{1+\eta}<\|e_1+e_2\|'_A=\|\tfrac1{1+\delta}a+\tfrac{\delta}{1+\delta}e_3\|'_A\le \tfrac1{1+\delta}\|a\|'_A+\tfrac{\delta}{1+\delta}\|e_3\|'_A=\tfrac1{1+\delta}\|a\|'_A+\tfrac{\delta}{1+\delta}$$and hence
$$
\begin{aligned}
\|a\|'_A&\ge (1+\delta)(\tfrac{1+\e}{1+\eta}-\tfrac{\delta}{1+\delta}\big).
\end{aligned}
$$
\end{proof}

\begin{proposition}\label{p:non} No based Banach space is almost $\FI_K$-universal for $K>1$.
\end{proposition}

\begin{proof} Assume that some based Banach space $(U,\|\cdot\|_U)$ is almost $\FI_K$-universal for some $K>1$. Choose any positive rational numbers $\eta<\e<\delta$ such that $$1+2\delta\le K,\;\;\delta\le 1,\;\;(1+\eta)(1+\e)\le(1+\delta)\;\;\mbox{and}\;\;\frac{1+\e}{1+\eta}>1+\frac{\delta}{1+\delta}.$$ Let $(A,\|\cdot\|_A)$ be the 3-dimensional $(1+2\delta)$-based Banach space with basis $\{e_1,e_2,e_3\}$, constructed in Example~\ref{ex}. Let $\Lambda$ be the based subspace of $A$ spanned by the vectors $e_1,e_2$. By Example~\ref{ex}(1), $\|x_1e_1+x_2e_2\|_A=\max\{|x_1|,|x_2|\}$ for any $x_1,x_2\in\IR$. Let  $\|\cdot\|_\Lambda'$ be the norm on $\Lambda\subset A$ defined by
$$\|x_1e_1+x_2e_2\|_\Lambda'=\max\{|x_1|,|x_2|,\tfrac{1+\e}2(|x_1|+|x_2|)\}.$$
It is clear that $$\|x\|_A\le \|x\|_\Lambda'\le(1+\e)\|x\|_A$$for any $x\in\Lambda$.

The definition of the norm $\|\cdot\|_\Lambda'$ implies that the based Banach space $(\Lambda,\|\cdot\|'_\Lambda)$ is $1$-based.
Since $U$ is almost $\FI_K$-universal, there exists an $\eta$-isometry $f:\Lambda\to U$ of the $1$-based Banach space $(\Lambda,\|\cdot\|_\Lambda')$ into the based Banach space $(U,\|\cdot\|_U)$.

 Taking into account that $f$ is an $\eta$-isometry, we conclude that
$$\tfrac1{1+\eta}\|x\|_A\le \tfrac1{1+\eta}\|x\|'_\Lambda<\|f(x)\|_U<(1+\eta)\|x\|'_\Lambda\le (1+\eta)(1+\e)\|x\|_A\le (1+\delta)\|x\|_A$$ for any $x\in \Lambda\setminus\{0\}$. 
This means that $f:\Lambda\to U$ is a $\delta$-isometry with respect to the norm  on $\Lambda$ induced by the norm $\|\cdot\|_A$. Since $(A,\|\cdot\|_A)$ is a $(1+2\delta)$-based Banach space, $(1+2\delta)\le K$, and the based Banach space $U$ is almost $\FI_K$-universal, the $\delta$-isometry $f$ extends to a $\delta$-isometry $\catB$-morphism $\bar f:A\to U$. The norm of the based Banach space $U$ induces the norm $\|\cdot\|_A'$ on $A$ defined by $\|x\|'_A=\|\bar f(x)\|_U$. It follows from $\bar f(e_3)\in \bar  f(\bas_A)\subset\bas_U$ that $\|e_3\|'_A=\|\bar f(e_3)\|_U=1$.

Taking into account that $f:(\Lambda,\|\cdot\|'_\Lambda)\to(U,\|\cdot\|_U)$ is an $\eta$-isometry, we conclude that
$$\|e_1+e_2\|'_A=\|f(e_1+e_2)\|_U>\tfrac1{1+\eta}\|e_1+e_2\|'_\Lambda.$$Then the conditions (3) and (4) of Example~\ref{ex} ensure that the point\newline $a=(1+\delta)\cdot(e_1+e_2)-\delta e_3$ has norms $\|a\|_A=1$ and
$$\|\bar f(a)\|_U=\|a\|_A'\ge (1+\delta)(\tfrac{1+\e}{1+\eta}-\tfrac{\delta}{1+\delta})>(1+\delta)\|a\|_A,$$
which means that $\bar f$ is not a $\delta$-isometry. So, $U$ is not almost $\FI_K$-universal. 
\end{proof}

On the other hand, the following theorem ensures that almost $\FI_1$-universal based Banach spaces do exist.  

\begin{theorem}\label{alrat}
Any $\RI_K$-universal rational $K$-based Banach space $X$ is almost $\FI_1$-universal.
\end{theorem}

\begin{proof}
We shall use the fact, that the norm of any finite-dimensional based Banach space can be approximated by a rational norm (which means that its unit ball coincides with the convex hull of finitely many points having rational coordinates in the basis).

Let $X$ be an $\RI_K$-universal rational $K$-based Banach space for some $K\ge 1$. To prove that $X$ is almost $\FI_1$-universal,  take any $\eps>0$, any finite-dimensional $1$-based Banach space $A$ and any $\eps$-isometry $\catB$-morphism $f:\Lambda\to X$ defined on a based subspace $\Lambda$ of $A$.  For every $F\subset \bas_{A}$ consider the linear operator $T_F:A\to A$ defined by the formula $T_F(\sum _{b\in \bas_{A}} x_b \cdot b)=\sum _{b\in  F} x_b \cdot b - \sum _{b\in \bas_A\setminus F} x_b \cdot b$ for $(x_b)_{b\in\bas_A}\in\mathbb R^{\bas_A}$. We recall that by $\|\cdot\|_A$ and $\|\cdot\|_\Lambda$ we denote the norms of the Banach spaces $A$ and $\Lambda$. 
The $\catB$-morphism $f$ determines a new norm $\| \cdot \|_\Lambda'$ on $\Lambda$, defined by $\| a \|_\Lambda'=\|f(a)\|_X$ for $a\in \Lambda$. Since $X$ is rational and $K$-based, $\|\cdot \|'_\Lambda$ is a  rational norm on $\Lambda$ such that $\|T_F(a)\|'_\Lambda\le K\cdot\|a\|'_\Lambda$ for every $a\in \Lambda$ and every subset $F\subset \bas_\Lambda:=\bas_A\cap\Lambda$. 
Taking into account that $f$ is an $\eps$-isometry, we conclude that $(1+\eps)^{-1}<\|a\|_\Lambda'<(1+\eps)$ for every $a\in \Lambda$ with $\|a\|_\Lambda=1$. By the compactness of the unit sphere in $\Lambda$, there exists a positive $\delta <\eps$ such that $(1+\delta )^{-1}<\|a\|'_\Lambda<(1+\delta )$ for every $a\in \Lambda$ with $\|a\|_\Lambda=1$.  This inequality implies $\frac{1}{1+\delta }B_\Lambda \subset B_\Lambda'\subset (1+\delta )B_\Lambda$, where $B_\Lambda=\{a\in \Lambda: \|a\|_\Lambda\leq 1\}$ and $B_\Lambda'=\{a\in \Lambda: \|a\|'_\Lambda\leq 1\}$ are the closed unit balls of $\Lambda$ in the norms $\| \cdot \|_\Lambda$ and $\|\cdot \|_\Lambda'$. Choose $\delta '$ such that  $\delta< \delta' <\eps$.

Let $B_A=\{a\in A: \|x\|_A\leq 1\}$ be the closed unit ball of the Banach space $A$. Choose a rational polyhedron $P$ in $A$ such that $P=-P$ and $$\tfrac{1}{1+\delta'}B_A\subset P \subset\tfrac1{1+\delta}B_A.$$ Since $A$ is a 1-based Banach space, $\bigcup_{F\subset\bas_A}T_F(B_A)=B_A$. So, we can replace $P$ by the convex hull of $\bigcup_{F\subset\bas_A}T_F(P)$ and assume that $T_F(P)=P$ for every $F\subset\bas_A$.

Consider the set
$$P':=B'_\Lambda\cup(\pm\bas_A)\cup P,$$
where $\pm\bas_A=\bas_A\cup(-\bas_A)$. The convex hull $B_A':=\operatorname{conv} (P')$ of $P'$ is a rational polyhedron in the based Banach space $A$. The rational polyhedron $B_A'$, being convex and symmetric, determines a rational norm $\| \cdot\|_A'$ on $A$ whose closed unit ball coincides with $B_A'$. By $A'$ we denote the Banach space $A$ endowed with the norm $\|\cdot\|'_A$.

Taking into account that $A$ is a 1-based Banach space and $P\subset \tfrac1{1+\delta}B_A$, we conclude that $$
P'=B'_\Lambda\cup(\pm\bas_A)\cup P\subset B'_\Lambda\cup (\pm\bas_A)\cup\tfrac1{1+\delta}B_A\subset(1+\delta)B_A
$$and hence $B_A'=\mathrm{conv}(P')\subset (1+\delta)B_A$. 

Let us show that $\|a\|'_A=\|a\|'_\Lambda$ for each $a\in \Lambda$, which is equivalent to the equality $B_A'\cap \Lambda=B_\Lambda'$. The inclusion $B_\Lambda'\subset B_A'\cap \Lambda$ is evident.
To prove the reverse inclusion $B_\Lambda'\supset B_A'\cap \Lambda$, consider the coordinate projection $$\pr_\Lambda:A\to \Lambda,\;\;\textstyle{\pr_\Lambda:\sum_{b\in\bas_A}x_bb\mapsto \sum_{b\in\bas_\Lambda}x_bb},$$onto the subspace $\Lambda$ of $A$. Taking into account that $A$ is a 1-based Banach space, we conclude that $\pr_\Lambda(B_A)=B_A\cap\Lambda=B_\Lambda$. 
Then $$
\begin{aligned}
\pr_\Lambda(P')&=\pr_\Lambda(B_\Lambda')\cup\pr_\Lambda(\pm\bas_A)\cup\pr_\Lambda(P)\subset\\
&\subset B'_\Lambda\cup\{0\}\cup(\pm\bas_\Lambda)\cup\pr_\Lambda(\tfrac1{1+\delta}B_A)=B'_\Lambda\cup \tfrac1{1+\delta}B_\Lambda=B'_\Lambda
\end{aligned}
$$as $\frac1{1+\delta}B_\Lambda\subset B_\Lambda'$. Consequently, $$B'_A\cap\Lambda\subset \pr_\Lambda(B'_A)=\pr_\Lambda(\conv(P'))=\conv(\pr_\Lambda(P'))\subset\conv(B'_\Lambda)=B_\Lambda',$$which completes the proof of the equality $B_A'\cap \Lambda=B'_\Lambda$.

The inclusion $\bas_A\subset B_A'$ implies that $\|b\|'_A\le 1$ for any $b\in\bas_A$. We claim that $\|b\|_A'=1$ for any $b\in\bas_A$. If $b\in\bas_\Lambda$, then 
$f(b)\in\bas_X$ and $\|b\|_A'=\|b\|_\Lambda'=\|f(b)\|_X=1$.
If $b\notin \bas_\Lambda$, then we can  consider the coordinate functional $\mathbf e^*_b\in A^*$ of $b$.
Since $\bas_A$ is a $1$-unconditional basis for the 1-based Banach space $A$, $\mathbf e^*_b(B_A)\subset[-1,1]$.  Then $$\mathbf e^*_b(P')\subset \mathbf e^*_b(B'_\Lambda)\cup\mathbf e^*_b(\pm\bas_A)\cup \mathbf e^*_b(B_A)\subset  \{0\}\cup \{-1,0,1\}\cup [-1,1]=[-1,1],$$ which means that the functional $\mathbf e^*_b$ has norm $\|\mathbf e^*_b\|'_{A^*}\le 1$ in the dual Banach space $(A')^*$.
Then $1=\mathbf e^*_b(b)\le\|\mathbf e^*_b\|'_{A^*}\cdot\|b\|'_{A}\le \|b\|'_{A}$ and hence $\|b\|'_{A}=1$. Therefore, the Banach space $A'$ endowed with the base $\bas_{A'}:=\bas_A$ is a based Banach space.

Next, we show that the based Banach space $A'$ is $K$-based. Indeed, for any $F\subset \bas_{A}$ we get 
$$
T_F(P')=T_F(B'_\Lambda)\cup T_F(\pm\bas_A)\cup T_F(P)\subset K\cdot B'_\Lambda\cup(\pm\bas _A)\cup P=[-K,K]\cdot P'$$ and hence 
\begin{multline*}
T_F(B'_A)=T_F(\conv(P'))=\conv(T_F(P'))\subset \conv([-K,K]\cdot P')=\\
=[-K,K]\cdot\conv(P')=K\cdot B_A',
\end{multline*} witnessing that the based Banach space $A'$ is $K$-based.

The inclusions $\frac{1}{1+\delta' }B_A\subset B_A'\subset (1+\delta )B_A$ imply the strict inequality 
\begin{equation}\label{eq1}
(1+\eps)^{-1}\|a\|_A < \|a\|_A' <(1+\eps)\|a\|_A
\end{equation} holding for all $a\in A\backslash \{0\}$.

Let $\Lambda'$ and $A'$ be the $K$-based Banach spaces $\Lambda$ and $A$ endowed with the new rational norms $\|\cdot\|_\Lambda'$ and $\|\cdot\|_A'$, respectively. It is clear that $\Lambda'\subset A'$. The definition of the norm $\|\cdot\|'_\Lambda$ ensures that $f:\Lambda'\to X$ is a $\BI$-morphism.
Using the $\RI_K$-universality of $X$, extend the isometry morphism $f:\Lambda'\to X$ to an isometry morphism $\bar f:A'\to X$. The inequalities~\eqref{eq1} ensure that $\bar f:A\to X$ is an $\eps$-isometry $\catB$-morphism from $A$, extending the $\eps$-isometry $f$.
This completes the proof of the almost $\FI_1$-universality of $X$.
\end{proof}

Combining Corollary~\ref{kun} with Theorem~\ref{alrat}, we get another model of the $\catB$-universal Pe\l czy\'nski's space $\U$.

\begin{corollary}\label{c:ru} Each $\RI_K$-universal rational $K$-based Banach space $\U_K$ is $\catB$-isomorphic to the $\catB$-universal Pe\l czy\'nski's space $\U$.
\end{corollary}

{\bf Acknowledgments.} Research of the second author was suppored by NCN grant DEC-2013/11/N/ST1/02963.

The authors express their sincere thanks to anonymous referees for valuable remarks and constructive criticism that eventually led us to the discovery of the non-existence of almost $\FI_K$-universal based Banach spaces for $K>1$ in Proposition~\ref{p:non}.

\bibliographystyle{amsplain}

\end{document}